\documentclass[reqno, a4paper]{amsart}

\usepackage{txfonts} 

\usepackage{amsmath, amssymb, bbm, amsthm,enumitem,  hyperref, color}
\usepackage{graphicx}

\numberwithin{equation}{section}
\newtheorem{thm}{Theorem}[section]

\newtheorem{prop}{Proposition}[section]
\newtheorem{remark}{Remark}[section]
\newtheorem{defn}{Definition}[section]

\newtheorem{lem}{Lemma}[section]

\newcommand{\E}{\mathbb E}

\title[Reciprocal specific relative entropy between continuous martingales]{Reciprocal specific relative entropy between continuous martingales}

\author{Julio Backhoff, Xin Zhang}

\begin{document}

\begin{abstract} 
We introduce a novel notion of divergence between continuous martingales; the reciprocal specific relative entropy. First, we motivate this definition from multiple perspectives. Thereafter, we solve the reciprocal specific relative entropy minimization problem over the set of win-martingales (used as models for prediction markets \cite{MR3096465}). Surprisingly, we show that the optimizer is the renowned neutral Wright-Fisher diffusion. We also justify that this diffusion is in a sense the most salient win-martingale, since it is uniquely selected when we suitably perturb the degenerate martingale optimal transport problem of variance minimization.
\end{abstract}

\thanks{This research was funded 
in whole or in part by the Austrian Science Fund (FWF) DOI 10.55776/P36835. X. Zhang is partially supported by the NSF Grant DMS-2508556.  Both authors thank Yuxing Huang for his help on numerical simulations; see Remark~\ref{eq:multidim}.}

\maketitle

\section{Introduction}

A divergence between martingale measures is a functional that quantifies how far alternative risk-neutral models are from a reference one. Ideally, this allows the decision-maker to analyze stability and model uncertainty while preserving the martingale constraint. Identifying a real-valued continuous martingale as a probability measures $\mathbb Q$ on the continuous path space $C([0,1];\mathbb{R})$, 
this note focuses on the study of the reciprocal specific relative entropy as a divergence between $\mathbb Q$ and the Wiener measure $\mathbb W$, given by
\begin{align*}
\mathfrak{h}( \mathbb Q \| \mathbb W):= \frac{1}{2}\mathbb E_{\mathbb Q}\left[\int_0^1[ \Sigma_t\log(\Sigma_t)+1-\Sigma_t]  dt \right] = \frac{1}{2}\mathbb E_{\mathbb Q}[H( \langle X\rangle\,|\,\text{Leb}_{[0,1]})],
\end{align*}
where $(\Sigma_t:=d\langle X\rangle_t/ dt)_{t \in [0,1]}$ is the instantaneous quadratic variation of the martingale measure $\mathbb Q$. Hence $\mathfrak{h}( \mathbb Q \| \mathbb W)$ measures the deviation between $\Sigma$ and $1$, the volatility of Brownian motion. (Equivalently, the deviation between the random measure $d\langle X\rangle_t$ and $dt$, which is induced by the quadratic variation of Brownian motion.)

Our motivation for studying the divergence $\mathfrak{h}$ is threefold:
\begin{enumerate}
    \item It acts as a selection criterion for some divergence optimization problem. As introduced in \cite{BaZh26}, the specific $p$-Wasserstein divergence measures the distance between martingales based on a power of their volatility. Though defined as a scaling limit, it often admits the explicit form
$$\textstyle\mathbb E_{\mathbb Q}\left[ \int_0^1\Sigma_t^{p/2}dt\right].$$
However, for the critical case $p=2$, the calibration problem of optimizing this quantity over martingales with given initial and terminal distributions yields non-unique solutions, due to It\^{o}'s isometry. In fact, every feasible martingale is optimal in this case. We establish that $\mathfrak{h}$ arises naturally (but under an integrability assumption) as the derivative of the specific $p$-Wasserstein divergence with respect to $p$ at $p=2$. Consequently, minimizing $\mathfrak{h}$ allows us to ``break the tie" and select a special (and unique) optimizer for $p=2$. 

\item This entropy represents the continuum limit of the relative entropy between discrete approximations of underlying martingales (Trinomial models), suggesting it is a natural entropic quantity for limit theorems involving discrete martingales when one passes to continuous-time. This fact had already been observed by Avellaneda et.\ al.\ \cite{AFHS97}.

\item It satisfies a reciprocal relationship with the better known specific relative entropy $$\textstyle h(\mathbb W \|\mathbb Q):= \frac{1}{2}\mathbb E_{\mathbb Q}\left [ \int_0^1\Sigma_t-\log\Sigma_t-1 \right ]=\frac{1}{2}\mathbb E_{\mathbb Q}[H(  \text{Leb}_{[0,1]}\,|\, \langle X\rangle)],$$ defined in \cite{Ga91}; That is, via a time-change argument, we show that $h(\mathbb W \|\mathbb Q)= \mathfrak{h}(\mathbb Q \| \mathbb W)$ for certain martingale measures, justifying at least partially the terminology ``reciprocal". In general, the real connection between $h(\mathbb W \|\mathbb Q)$ and $\mathfrak{h}(\mathbb Q \| \mathbb W)$ lies in flipping the roles of $\langle X\rangle$ and $\text{Leb}_{[0,1]}$ in the definitions in terms of expected entropies.
\end{enumerate}
 
We explain in more detail these motivations in Section \ref{sec:reverse_motivation}.  To further support the choice of the reciprocal specific relative entropy as a measure of divergence between a martingale and Brownian motion, the rest of the article is devoted to an example of a calibration (i.e., optimization) problem where the reciprocal specific relative entropy serves as a loss function. 
This is also the core contribution of this work. Concretely, we fully solve the associated win-martingale optimization problem.  That is, letting $\mathcal{M}_{0,x_0}^{win}$ denote the class of martingales that start at $x_0\in(0,1)$ and terminate at the boundary values $\{0,1\}$, we consider the optimization problem \begin{equation}\label{eq:intromin}\inf_{\mathbb Q \in \mathcal{M}_{0,x_0}^{win}} \mathfrak{h}(\mathbb Q \| \mathbb W)\end{equation}These win-martingales serve as fundamental models for prediction markets \cite{MR3096465}, where the asset price must eventually converge to a binary outcome (win or loss). By now it is understood that calibration problems do not generally admit an explicit solution unless the terminal prescribed marginal is of this form, and this is the reason why we chose to test the potential of the reciprocal specific relative entropy in this setting.

By framing the aforementioned win-martingale minimization as a stochastic control problem, we derive the associated Hamilton-Jacobi-Bellman (HJB) equation. Our main result (Theorem \ref{thm}) establishes that the unique optimizer of \eqref{eq:intromin} is, rather surprisingly, the (scaled) neutral Wright-Fisher diffusion\footnote{`Neutral' here refers to the fact that there is no drift. We also write `scaled' because the diffusion has been time-changed to be indexed by $[0,1]$.}
$$dX_t=\sqrt{\frac{X_t(1-X_t)}{1-t}} \,dB_t, \,\,\,\, X_0=x_0.$$
This process is renowned in population genetics for modeling gene frequency dynamics. For us, it emerges naturally as the prediction market (i.e., win-martingale) that minimizes uniquely the reciprocal specific entropy. We prove this result by verifying that the candidate value function is the unique viscosity solution to the derived HJB equation. As an interesting intermediate step, we obtain that the squared volatility process $\frac{X_t(1-X_t)}{1-t}$ is a martingale on $[0,1)$, a fact that we haven't been able to find in the classical literature on the subject.  In addition, we check the necessary integrability conditions (see Point (1) above) to link the value function rigorously to the derivative of the specific $p$-Wasserstein divergence minimization problem for $p=2$.  More precisely, we prove that 
\begin{align*}
    \inf_{\mathbb Q \in \mathcal{M}_{0,x_0}^{win}} \mathfrak{h}(\mathbb Q \| \mathbb W) =\frac{d}{dp}\Big|_{p=2} \inf_{\mathbb Q \in \mathcal{M}_{0,x_0}^{win}} \mathbb E_{\mathbb Q}\left[ \int_0^1 \Sigma_t^{p/2} \, dt \right].
\end{align*}
This means that the neutral Wright-Fisher diffusion is selected, among all competing win-martingales (all of which are optimal for $p=2$), when we let $p$ approach $2$.

\subsection{Related literature}

The specific relative entropy was first introduced and interpreted as a rate function for a large deviation principle by Gantert \cite{Ga91}. More recently \cite{BaBe24,MR4651162}  obtained an explicit formula for this quantity between time-homogeneous Markov martingales, and F\"{o}llmer \cite{MR4433813} extended Gantert's results and established a Talagrand-type inequality between semimartingale laws using this object. In \cite{BBexciting,BWZ24} the specific relative entropy was used to solve an open question by Aldous (see \cite{2023arXiv230607133G} for an alternative point of view and solution). This concept of specific relative entropy appeared in mathematical finance in the works of Avellaneda et al \cite{AFHS97}, Cohen and Dolinsky \cite{CoDo22}, and also Dolinsky and Zhang \cite{DoZh25}.The recent articles \cite{BeChHoLoVi24,BeChLo24} by Benamou, Chazareix, Hoffmann and Loeper have suggested the use of the specific relative entropy as a regularization device for the problem of model calibration in finance, and studied the convergence of discrete- to continuous-time. Following this line of research, in \cite{BaZh26} the authors introduced the specific $p$-Wasserstein divergence, and analyzed its optimization problem.

Win-martingales have been proposed as models for prediction markets (cf.\ \cite{MR3096465}), and optimization problems over the set of win-martingales were proposed by Aldous \cite{aldous_winmartingale} in connection to the aforementioned open question. Such optimization problems are particular instances of martingale optimal transport, a subject that has been extensively studied in the recent years. Following \cite{BeHePe12, BoNu13,GaHeTo13,HoNe12}
martingale versions of the classical transport problem are often considered due to applications in  mathematical finance but admit further applications, e.g.\ to the Skorokhod problem \cite{BeCoHu14, BeNuSt19}. In analogy to classical optimal transport, necessary and sufficient conditions for optimality have been established for martingale transport (MOT) problems in discrete time (\cite{BeJu21, BeNuTo16}) but not so much is known for the continuous time problem. Notable exceptions are \cite{BWZ24,BBexciting,backhoff2020martingale,2023arXiv230611019B,2023arXiv230607133G,GuLo21,huesmann2019benamou,Lo18,tan2013optimal}.

\section{Reciprocal specific relative entropy}\label{sec:reverse_motivation}

In this section, we provide different motivations for the reciprocal specific relative entropy $\mathfrak{h}$. First, it can be viewed as the derivative of the specific $p$-Wasserstein divergence defined in \cite{BaZh26} as $p \to 2$. Second, it is the limit of relative entropy between discrete approximations of underlying martingales. Third, by a change of variable argument, for a certain class of martingale measures $\mathbb Q$ it can be show that $h(\mathbb Q \|\mathbb W)= \mathfrak{h}(\mathbb W \| \mathbb Q)$. That is why we call it the reciprocal specific relative entropy. 

\subsection{Limit of the specific $p$-Wasserstein divergence} 

Suppose $\mathbb Q \in \mathcal{P}(C([0,1];\mathbb R))$ is a one-dimensional martingale measure. As defined in \cite{BaZh26}, its specific $p$-Wasserstein divergence with respect to the constant martingale $\mathbb P_{\delta}$, i.e. a martingale with zero quadratic variation, is given explicitly by the formula
\begin{align}\label{eq:swp}
     \mathbb E_{\mathbb Q}\left[\int_0^1  |\sigma_t|^p \,dt \right]=\mathbb E_{\mathbb Q}\left[\int_0^1  \Sigma_t^{p/2} \,dt \right],
\end{align}
up to a constant factor, where $(\sigma_t^2:=\Sigma_t)_{t\in[0,1]}$ is the instantaneous quadratic variation process under the measure $\mathbb Q$. Taking derivative w.r.t.\ $p$ at $p=2$ and assuming one can exchange integration and derivative we obtain (almost) the expression of  the reciprocal specific relative entropy
\begin{align}\label{eq:rse}
   \frac{1}{2} \mathbb E_{\mathbb Q}\left[\int_0^1 \Sigma_t \log(\Sigma_t) \, dt \right].
\end{align}

For any $\mu, \nu \in \mathcal{P}_2(\mathbb R)$ with $\mu \leq_c \nu$ (i.e.,\ having finite second-order moment and being in convex order), consider the family of martingale optimal transport problems 
\begin{align}\label{eq:MOTp}
MT_p(\mu,\nu):= \inf_{\mathbb Q \in \mathcal{M}(\mu,\nu)} \mathbb E_{\mathbb Q}\left[\int_0^1 |\sigma_t|^p \,dt \right], \quad p > 2. 
\end{align}
For $p=2$, It\^{o}'s isometry yields $MT_2(\mu,\nu)=\int x^2 \, d(\nu-\mu)$, and any feasible martingale in $\mathcal{M}(\mu,\nu)$ is an optimizer. We want to single out a particular one by considering the optimization problem 
\begin{align}\label{eq:MOTre}
   MT_{re}(\mu,\nu):=\inf_{\mathbb Q \in \mathcal{M}(\mu,\nu)}  \frac{1}{2}  \mathbb E_{\mathbb Q}\left[\int_0^1 \Sigma_t \log(\Sigma_t) \, dt \right],
\end{align}
which can be interpreted as the derivative of $MT_p(\mu,\nu)$ w.r.t.\ $p$ at $p=2$. To wit:
\begin{lem}\label{lem:swplimit}
We always have $$\liminf_{p \searrow 2} \frac{MT_p(\mu,\nu)-MT_2(\mu,\nu)}{p-2} \geq MT_{re}(\mu,\nu).$$ If an optimizer $\mathbb Q^*$ of \eqref{eq:MOTre} satisfying $\mathbb E_{\mathbb Q^*} \left[ \int_0^1 |\sigma_t|^{2+\epsilon} \, dt\right]$ for some $\epsilon \in (0,1)$ exists, then  
    \begin{align*}
    \lim\limits_{p \searrow 2} \frac{MT_p(\mu,\nu)-MT_2(\mu,\nu)}{p-2}= MT_{re}(\mu,\nu).
    \end{align*}
\end{lem}
\begin{proof}

As $p \mapsto |\sigma|^p$ is convex, 
\begin{align*}
   |\sigma|^p \geq |\sigma|^2+\frac{p-2}{2}|\sigma|^2 \log (\sigma^2),
\end{align*}
which yields that 
\begin{align*}
    \frac{|\sigma|^p-|\sigma|^2}{p-2} \geq \frac{1}{2} |\sigma|^2 \log(\sigma^2). 
\end{align*}
Minimizing both sides over martingale measures, we obtain that $$\frac{MT_p(\mu,\nu)-MT_2(\mu,\nu)}{p-2} \geq MT_{re}(\mu,\nu),$$
and hence $$\liminf_{p \searrow 2} \frac{MT_p(\mu,\nu)-MT_2(\mu,\nu)}{p-2} \geq MT_{re}(\mu,\nu).$$

On the other hand, applying Fatou's lemma 
\begin{align}\label{eq:limitfatou}
   \frac{1}{2}\mathbb E_{\mathbb Q^*}\left[ \int_0^1 \sigma^2_t \log \sigma^2_t \, dt \right] &= \mathbb E_{\mathbb Q^*}\left[\int_0^1 \limsup\limits_{p \searrow 2}  \frac{\sigma_t^p-\sigma_t^2 }{p-2}dt\right] \notag \\
    & \geq \limsup\limits_{p \searrow 2} \frac{\mathbb E_{\mathbb Q^*}\left[\int_0^1 \sigma_t^p-\sigma_t^2  \, dt \right]}{p-2},
\end{align}
where we used that 
\[\frac{|\sigma|^p-|\sigma|^2 }{p-2} \leq \frac{|\sigma|^{2+\epsilon}-|\sigma|^2 }{\epsilon}, \quad p \in (2, 2+\epsilon),\]
and thus the l.h.s.\ is bounded from above by an integrable function, by assumption. Therefore $$\limsup_{p \searrow 2} \frac{MT_p(\mu,\nu)-MT_2(\mu,\nu)}{p-2} \leq MT_{re}(\mu,\nu).$$
\end{proof}

\begin{remark}
 We comment that the integrability assumption in Lemma~\ref{lem:swplimit} is 
 necessary in \eqref{eq:limitfatou} when applying Fatou's lemma for a martingale measure $\mathbb Q$. To wit, let us take
$\sigma_t^2 = 1/\left(t \left( \ln\left(e/t\right) \right)^3\right)$, and $\mathbb Q$ to be the law of process $\left(\int_0^t \sigma_s \,dB_s\right)_{t \in [0,1]}$. It can be easily checked that $\int_0^1 \sigma_t^2 \log(\sigma_t^2) \, dt < +\infty$, but $\int_0^1 \sigma_t^{2(1+\epsilon)} \, dt = +\infty$ for any $\epsilon \in (0,1)$. Then we have $\mathbb E_{\mathbb Q}\left[ \int_0^1 \sigma^2_t \log \sigma^2_t \, dt \right]<+\infty$, but $\mathbb E_{\mathbb Q}\left[\int_0^1 \sigma_t^p \,dt \right]=+\infty$ for any $p>2$. Therefore \eqref{eq:limitfatou} fails for $\mathbb Q$. Later we will verify that the Wright-Fisher diffusion satisfies this integrability assumption. 
\end{remark}

\subsection{Limit of Trinomial Models}

Let $h:=1/N>0$ and $\sigma,\sigma_0>0$ be given. Take now any $\bar \sigma>\sigma,\sigma_0$. The simplest martingale evolving in a trinomial tree that behaves qualitatively like $(\sigma B_t)_t$ is $(M^\sigma_t:t\in\{0h,1h,...,Nh\})$ with independent increments $M^\sigma_{(k+1)h}-M^\sigma_{kh}\in\{-\bar\sigma\sqrt{h},0,\bar\sigma\sqrt{h}\}$ such that
$$\mathbb P(M^\sigma_{(k+1)h}-M^\sigma_{kh}=\bar\sigma\sqrt{h})=\mathbb P(M^\sigma_{(k+1)h}-M^\sigma_{kh}=-\bar\sigma\sqrt{h})= \frac{\sigma^2}{2\bar\sigma^2}. $$
This way the variance of the increment of $M^\sigma$ is precisely $\sigma^2h$. It follows that 
$$(\bar\sigma \sqrt{h})^2H(M^\sigma|M^{\sigma_0})= \sigma^2\log \frac{\sigma^2}{\sigma_0^2}+\left(\bar\sigma^2-\sigma^2\right )\log \frac{\bar\sigma^2-\sigma^2}{\bar\sigma^2-\sigma_0^2}=\sigma^2\log \frac{\sigma^2}{\sigma_0^2}+\sigma_0^2-\sigma^2 +o(1),$$
for $\bar\sigma$ large. Specializing to the case $\sigma_0=1$ we conclude that in the limit where $h$ is going to zero and $\bar\sigma$ to infinity, 
$$(\bar\sigma \sqrt{h})^2H(M^\sigma|M^{1})\to \sigma^2\log (\sigma^2)-\sigma^2+1=H(\lambda^\sigma|\lambda), $$
with $\lambda$ Lebesgue measure on $[0,1]$, $d\lambda^\sigma(x)=\sigma^2d\lambda(x)$, and $H(q|p):\int\{\frac{dq}{dp}\log \frac{dq}{dp} - \frac{dq}{dp} +1\}  dp $ is an extension of the relative entropy beyond probability measures. 
This gives us a quantity that does not depend on the space-grid parameter $\bar\sigma$ 
(similar as in Avellaneda et al \cite{AFHS97}). This quantity is precisely the reciprocal specific relative entropy between $(\sigma B_t)_t$ and $(B_t)_t$.

\subsection{Reversing the roles of $\mathbb Q$ and $\mathbb W$}

Suppose $\mathbb Q$ is the law of the SDE $dX_t=\sigma(X_t) \, dB_t$ with say $X_0=0$. Then 
\begin{align*}
    h(\mathbb W || \mathbb Q)=\frac{1}{2}\mathbb E_{\mathbb W} \left[ \int_0^1 \frac{1}{\sigma(Y_t)^2}-\log\left(\frac{1}{\sigma(Y_t)^2} \right)-1 \,dt \right]. 
\end{align*}
Noting $X_{\langle X \rangle^{-1}_t}$ is a Brownian motion under $\mathbb Q$, we have
\begin{align*}
    h(\mathbb W || \mathbb Q)=\frac{1}{2}\mathbb E_{\mathbb Q} \left[ \int_0^1 \frac{1}{\sigma(X_{\langle X \rangle^{-1}_t})^2}-\log\left(\frac{1}{\sigma(X_{\langle X \rangle^{-1}_t})^2} \right)-1 \,dt\right].
\end{align*}
Changing variable $t=\langle X\rangle_s$, we have 
\begin{align*}
    &\int_0^1 \frac{1}{\sigma(X_{\langle X \rangle^{-1}_t})^2}-\log\left(\frac{1}{\sigma(X_{\langle X \rangle^{-1}_t})^2} \right)-1 \,dt \\
    &= \int_0^{\langle X \rangle^{-1}(1)} \left( \frac{1}{\sigma(X_s)^2}-\log\left(\frac{1}{\sigma(X_s)^2} \right)-1 \right) \sigma(X_s)^2 \,ds. 
\end{align*}
Therefore, we have that 
\begin{align*}
    h(\mathbb W || \mathbb Q)&=\frac{1}{2} \mathbb E_{\mathbb Q}\left[ \int_0^{\langle X \rangle^{-1}(1)} \left( \frac{1}{\sigma(X_s)^2}-\log\left(\frac{1}{\sigma(X_s)^2} \right)-1 \right) \sigma(X_s)^2 \,ds\right] \\
    &=\frac{1}{2} \mathbb E_{\mathbb Q}\left[ \int_0^{\langle X \rangle^{-1}(1)}  1+\sigma(X_s)^2\log\left(\sigma(X_s)^2 \right)-\sigma(X_s)^2 \,ds\right].
\end{align*}
As a result, if (and only if) $\langle X\rangle_1=1$ $\mathbb Q$-a.s., then 
$$h(\mathbb W || \mathbb Q) = \mathfrak{h}( \mathbb Q \| \mathbb W).$$

\section{Win-martingale optimization}
In this section, we would like to minimize $$\mathfrak{h}(\mathbb Q | \mathbb W)=\frac{1}{2}\mathbb E_{\mathbb Q} \left[\int_0^1 \sigma_u^2 \log(\sigma_u)^2+1-\sigma_u^2 \, du \right] $$ as $\mathbb Q$ varies over the class of win martingales. Denote by $\mathcal{M}_{t,x}^{win}$ the set of win-martingales that start with $x \in (0,1)$ at time $t \in [0,1]$, that is, continuous martingales that admit an absolutely continuous quadratic variation, start at $x$ at time zero and finish in either $0$ or $1$ at time one. As the integration of $(1-\sigma_u^2)$ is independent of the choice of $\mathbb Q \in \mathcal{M}_{t,x}^{win}$, we consider an equivalent optimization problem
\begin{align}\label{eq:WFvalue}
     v(t,x)=  \inf_{\mathbb Q \in \mathcal{M}_{t,x}^{win}} \frac{1}{2}\mathbb E_{\mathbb Q}\left[\int_t^1 \sigma_u^2 \log(\sigma^2_u) \,du  \right],
\end{align}
which satisfies the HJB equation by direct computation
\begin{align*}
    -\partial_t v(t,x)= \inf_{\Sigma \geq 0} \frac{1}{2} \left\{ \Sigma \log(\Sigma)+ \Sigma \partial_x^2 v(t,x) \right\}.
\end{align*}
As $\mathbb R_+ \mapsto \Sigma \log(\Sigma)+ \Sigma \partial_x^2(t,x)$ is strictly convex, the first order condition gives the optimal 
\begin{align}\label{eq:optimalSigma}
    \Sigma^*_t(x)= \exp( -\partial_x^2 v(t,x)-1).
\end{align}
Hence, one can simplify the Hamiltonian, and obtain the HJB equation
\begin{align}\label{eq:HJB}
\begin{cases}
  -\partial_t v(t,x) = - \frac{1}{2} e^{- \partial_x^2 v(t,x)-1}  \\
  \hspace{12pt} v(1,x) = +\infty,   \quad  \quad \quad \hspace{19pt} x \in  (0,1) \\
  \hspace{14pt} v(t,x) = 0,  \quad \quad \quad \quad \ \  (t,x) \in [0,1) \times \{0,1\}.
\end{cases}
\end{align}

By a scaling argument as in Lemma~\ref{lem:terminal} below, we have that $ v(t,x)=v(0,x)-\frac{1}{2}\log(1-t)x(1-x)$, and $v(0,x)$ satisfies 
\begin{align}\label{eq:elliptic}
    x(1-x)=\exp\left(-\partial_x^2 v(0,x)-1 \right),
\end{align}
with the boundary condition $v(0,0)= v(0,1)=0$. It can be verified that
\begin{align*}
    f(x):=-\left( \frac{1}{4}x^2 \log(x^2)+\frac{1}{4}(1-x)^2 \log ((1-x)^2)+x(1-x) \right).
\end{align*}
is a solution to \eqref{eq:elliptic}. Therefore, it can be checked that
\begin{align}\label{eq:candidate}
\bar v(t,x): =&  -\left( \frac{1}{4}x^2 \log(x^2)+\frac{1}{4}(1-x)^2 \log ((1-x)^2)+x(1-x) \right) \\
&-\frac{1}{2}\log(1-t)x(1-x) \notag 
\end{align}
solves \eqref{eq:WFvalue}. Moreover, \eqref{eq:optimalSigma} and \eqref{eq:elliptic} imply that $\Sigma^*_t(x)=\frac{x(1-x)}{1-t}$, and a candidate of optimizer is given by the scaled neutral Wright-Fisher diffusion \begin{align}\label{eq:WFdiffusion}
\begin{cases}
dM_s=\sqrt{\frac{M_s(1-M_s)}{1-s}} \, dB_s\,\, (s>t), \\
\, \ \, M_t=x.
\end{cases}
\end{align}

Recalling one motivation of the reciprocal specific relative entropy in Lemma~\ref{lem:swplimit}, we will also show that the value function $v$ is the derivative of a family of martingale optimal transport problems. Let us summarize our main result, whose proof will be divided into several parts and presented later:

\begin{thm}\label{thm}
    The unique optimizer of \eqref{eq:WFvalue} is the scaled neutral Wright-Fisher diffusion \eqref{eq:WFdiffusion}, the value function is given explicitly via  $v(t,x)=\bar v(t,x)$. Furthermore, we have \begin{align}\label{eq:derivativeswp}
    v(0,x)= \frac{d}{dp} \Big|_{p=2} MT_p(\delta_x, x\delta_1+(1-x)\delta_0).
    \end{align}
\end{thm}

\begin{remark}
    There is another way of finding the optimizer by calculus of variation. According to the first order (necessary optimality) condition as in \cite{BBexciting}, we know that at the optimum
$$ t \mapsto \sigma_t^2 (\log(\sigma_t^2)+1) - \sigma_t^2 \log (\sigma_t^2)=\sigma_t^2 $$
is a martingale, which implies that 
\begin{align*}
    0=\partial_t \Sigma_t(x)+ \frac{1}{2} \Sigma_t(x) \, \partial_x^2\Sigma_t(x)
\end{align*}
Assuming $\Sigma_t(x)=\frac{g(x)}{1-t}$, this reduces to 
\begin{align*}
    0=\frac{g(x)}{(1-t)^2}+\frac{1}{2}\frac{g''(x)f(x)}{(1-t)^2}, 
\end{align*}
and hence $g''(x)=-2$. We get that $g(x)=x(1-x)$, and thus after time scaling the optimal martingale becomes the neutral Wright-Fisher diffusion. Incidentally, reading this argument from the bottom up shows that $(\Sigma_t)_t$ is a martingale for the instantaneous quadratic variation of the scaled neutral Wright-Fisher diffusion. 
\end{remark}

 We first verify that the value of the scaled neutral Wright-Fisher diffusion is exactly $\bar v(t,x)$:

\begin{prop}\label{prop:WF}
    Suppose $\bar{\mathbb Q} $ is the distribution of the scaled neutral Wright-Fisher diffusion \eqref{eq:WFdiffusion}. Then we have the value associated with $\bar{\mathbb Q}$, i.e., 
    \begin{align*}
    \bar v(t,x)= &\frac{1}{2}\mathbb E_{\bar{\mathbb Q}} \left[\int_t^1 \Sigma_u \log(\Sigma_u) \, du \right] \\
    = & -\left( \frac{1}{4}x^2 \log(x^2)+\frac{1}{4}(1-x)^2 \log ((1-x)^2)+x(1-x) \right)-\frac{1}{2}\log(1-t) x(1-x). 
    \end{align*}
Moreover, $\bar v$ satisfies the HJB equation \eqref{eq:HJB}, and $ -\infty < v(t,x) \leq \bar v(t,x)<+\infty$ for any $(t,x) \in [0,1) \times [0,1]$.
\end{prop}
\begin{proof}
By the definition of Wright-Fisher diffusion, we have $\Sigma_u= \frac{M_u(1-M_u)}{1-u}$, where 
$M_u$ satisfies \eqref{eq:WFdiffusion}. Thus $\log(\Sigma_u)= \log(M_u)+\log(1-M_u)- \log(1-u)$. Define a function $f(x)=\frac{1}{4}x^2\log(x^2)-\frac{3}{4} x^2$, $x \in \mathbb R_+$, whose second order derivative is $ \log(x)$. According to It\^{o}'s formula, we have that 
\begin{align*}
    f(M_1)-f(M_t)= \int_t^1 f'(M_u) \,dM_u+ \frac{1}{2} \int_t^1 \log(M_u) \Sigma_u \,du.
\end{align*}
As $f'(x)$ is bounded for $x \in [0,1]$, the stochastic integral above is actually a martingale. Therefore,
\begin{align*} 
\frac{1}{2}\mathbb E \left[\int_t^1 \Sigma_u \log(M_u) \,du \right]= \mathbb E[ f(M_1)-f(M_t)]=-\frac{3x}{4}-\frac{1}{4}x^2 \log(x^2)+\frac{3}{4}x^2.
\end{align*}
Similarly, we obtain that 
\begin{align*}
    \frac{1}{2}\mathbb E \left[\int_t^1 \Sigma_u \log(1-M_u) \,du \right]& =\mathbb E[f(1-M_1)-f(1-M_t)] \\
    &=-\frac{3(1-x)}{4}- \frac{1}{4}(1-x)^2\log((1-x)^2)+\frac{3}{4}(1-x)^2. 
\end{align*}
Since $\Sigma_u$ is a martingale, $\mathbb E[\Sigma_u ]=\frac{x(1-x)}{1-t}$ for $u \in [t,1]$, so that 
\begin{align*}
    -\frac{1}{2} \mathbb E \left[\int_t^1 \Sigma_u \log(1-u) \, du \right]= -\frac{x(1-x)}{2}(\log(1-t)-1). 
\end{align*}
Summing them up, we get 
\begin{align*}
\bar v(t,x)=  -\left( \frac{1}{4}x^2 \log(x^2)+\frac{1}{4}(1-x)^2 \log ((1-x)^2)+x(1-x) \right)-\frac{1}{2}\log(1-t) x(1-x). 
\end{align*}

By direct computation, one can verify that $\bar v$ satisfies the HJB equation \eqref{eq:HJB}. As \eqref{eq:WFdiffusion} is a win-martingale, $v(t,x) \leq \bar v(t,x)<+\infty$ over $(t,x) \in [0,1) \times [0,1]$. 
Finally, since $\Sigma \log(\Sigma)-\Sigma+1 \geq 0$ for any $\Sigma \geq 0$, we have that for any $\mathbb Q \in \mathcal{M}_{t,x}^{win}$
\begin{align*}
    \frac{1}{2}\mathbb E_{{\mathbb Q}} \left[\int_t^1 \Sigma_u \log(\Sigma_u) \, du \right] \geq \frac{1}{2} \mathbb E_{\mathbb Q} \left[\int_t^1 \Sigma_u -1 \, du \right]=\frac{x(1-x)-(1-t)}{2}.
\end{align*}
Hence $v(t,x) > -\infty$ for all $(t,x) \in [0,1) \times [0,1]$.

\end{proof}

For the proof of uniqueness, we consider a class of equations indexed by $\lambda \in (0,1)$ 
\begin{align}\label{eq:gHJB}
\begin{cases}
  -\partial_t v(t,x) = - \frac{1}{2} e^{- \partial_x^2 v(t,x)-1} \, \, \, (t,x) \in [\lambda,1) \times (0,1), \\
  \hspace{12pt} v(1,x) = +\infty,   \quad  \quad \quad \hspace{6pt} (1,x) \in \{1\} \times  (0,1), \\
  \hspace{14pt} v(t,x) = 0,  \quad \quad \quad \quad \ \ \,\,  (t,x) \in [\lambda,1) \times \{0,1\}.
\end{cases}
\end{align}

\begin{defn}
    A continuous function $u_1:[\lambda,1) \times [0,1]\to \mathbb R$ is said to be a viscosity sub-solution to \eqref{eq:gHJB} if 
    \begin{enumerate}
        \item[(i)] $u_1(t,x) \leq 0$ for any $(t,x) \in [\lambda,1) \times \{0,1\}$; 
        \item[(ii)] For any test function $\phi \in C^2 ([\lambda,1) \times [0,1]; \mathbb R)$, and any local maximum $(t_0,x_0) \in [\lambda,1) \times (0,1)$ of $u_1-\phi$, we have 
        \begin{align*}
        -\partial_t \phi(t_0,x_0) \leq -\frac{1}{2}\exp(-\partial_x^2\phi(t_0,x_0)-1).
        \end{align*}
    \end{enumerate}
A continuous function $u_2:[\lambda,1) \times [0,1] \to \mathbb R$ is said to be a viscosity super-solution to \eqref{eq:gHJB} if 
\begin{enumerate}
    \item [(i)] $u_2(t,x) \geq 0$ for any $(t,x) \in [\lambda,1) \times \{0,1\}$, and $\limsup_{(t,y) \to (1,x)} u_2 (t,y) =+\infty$, $\forall x \in (0,1)$;
    \item [(ii)]
    For any test function $\phi \in C^2 ([\lambda,1) \times [0,1]; \mathbb R)$, and any local minimum $(t_0,x_0) \in [\lambda,1) \times (0,1)$ of $u_2-\phi$, we have 
        \begin{align*}
        -\partial_t \phi(t_0,x_0) \geq -\frac{1}{2}\exp(-\partial_x^2\phi(t_0,x_0)-1).
        \end{align*}
\end{enumerate}
\end{defn}

\begin{lem}\label{lem:terminal}
 The value function $v$ of \eqref{eq:WFvalue} satisfies that 
 \begin{align*}
     v(t,x)+\frac{1}{2}\log(1-t)x(1-x)= v(s,x)+\frac{1}{2}\log(1-s)x(1-x), \quad \forall t,x \in [0,1), \, x \in [0,1]. 
 \end{align*}
\end{lem}
\begin{proof}
      For any $\epsilon>0$, suppose $(M^t_u)_{u \in [t,1]}$ with distribution $\mathbb Q^t \in \mathcal{M}_{t,x}^{win}$ is a win-martingale, whose density of quadratic variation we denote $\Sigma^t$, such that 
\begin{align*}
    v(t,x) \leq \frac{1}{2}\E_{\mathbb Q^t}\left[\int_t^1 \Sigma^t_u \log (\Sigma^t_u)\, du \right] <v(t,x)+\epsilon. 
\end{align*}
Let us take $M^s_u= M^t_{t+(u-s)(1-t)/(1-s)}$ for $u \in [s,1]$, and denote its distribution by $\mathbb Q^s$ and the density of the quadratic variation by $\Sigma^s$. Therefore
\begin{align*}
v(s,x) \leq & \frac{1}{2} \E_{\mathbb Q^s} \left[\int_s^1 \Sigma^s_u \log( \Sigma^s_u ) \,du \right] \\
=&  \frac{1}{2}\E_{\mathbb Q^s} \left[\int_s^1 \frac{1-t}{1-s} \Sigma^t_{t+(u-s)(1-t)/(1-s)} \log \left (\frac{1-t}{1-s} \Sigma^t_{t+(u-s)(1-t)/(1-s)} \right ) \,du \right]\\
=& \frac{1}{2}\E_{\mathbb Q^t}\left[\int_t^1 \Sigma^t_u \log\left(  \Sigma^t_{u} \right) \,du \right] +\frac{1}{2} \left( \log(1-t)-\log(1-s)\right)  \E_{\mathbb Q^t} \left[\int_t^1      \Sigma_u^t \,du  \right]\\
\leq & v(t,x)+\epsilon+\frac{1}{2} (\log(1-t)-\log(1-s)) x(1-x). 
\end{align*}
Letting $\epsilon \to 0$, we get half of the equality. Exchanging the roles of $t$ and $s$, we finish proving the result. 
\end{proof}

\begin{prop}\label{prop:comparison}
    The value function $v$ defined by the control problem \eqref{eq:WFvalue} is the unique viscosity solution to \eqref{eq:HJB}. 
\end{prop}
\begin{proof}
    By a standard argument, the value function $v$ is a viscosity solution to 
    \begin{align*}
        -\partial_t v(t,x)= -\frac{1}{2}e^{-\partial_x^2 v(t,x) -1}. 
    \end{align*}
    Thanks to Lemma~\ref{lem:terminal}, $v$ satisfies the terminal condition $v(1,x)=+\infty$ for $x \in (0,1)$. Moreover, for any win-martingale $M$ with $M_t \in \{0,1\}$, $M_u=M_t$ for $u \in [t,1]$ and thus $\Sigma_u=0$ and $\Sigma_u \log(\Sigma_u)=0$ for $u \in [t,1]$. Therefore $v(t,0)=v(t,1)=0$ for $t \in [0,1)$, and $v$ satisfies the boundary condition. Thus $v$ is a viscosity solution to \eqref{eq:HJB}. 

    Let us now prove the uniqueness. Take $u_1$ and $u_2$ to be a continuous viscosity sub-solution and super-solution respectively. For any small $\lambda>0$, define $u_1^{\lambda}(t,x)=u_1(t-\lambda,x)$ for $(t,x) \in [\lambda,1) \times [0,1]$. It can be easily verified $u_1^{\lambda}$ is a continuous viscosity sub-solution to \eqref{eq:gHJB}. As $u_1$ is continuous over $[0,1) \times [0,1]$, $u^{\lambda}_1$ is uniformly continuous over the compact domain $[\lambda,1] \times [0,1]$. So the maximum of $u_1^{\lambda} - u_2$ must be obtained before terminal time $1$. Now by a standard argument of doubling variable $u_1^{\lambda} \leq u_2$ over $[\lambda,1) \times [0,1]$, which holds for any $\lambda \in (0,1)$. Thanks to the continuity of $u_1$, we get that $u_1=\lim\limits_{\lambda \to} u_1^{\lambda} \leq u_2$ over $[0,1) \times [0,1]$, and thus finish proving the comparison principle. 
\end{proof}

Let us apply Lemma~\ref{lem:swplimit} to obtain \eqref{eq:derivativeswp}. For that purpose, we need to verify the integrability of the scaled neutral Wright-Fisher diffusion. 

\begin{prop}\label{prop:moment}
Suppose $\bar{\mathbb Q}$ is the distribution of the scaled neutral Wright-Fisher diffusion. Then 
\begin{align}\label{eq:momentbound}
\mathbb E_{\bar{\mathbb Q}} \left[ \int_0^1 \Sigma_t^{1.5} \, dt \right] <+\infty. 
\end{align}
\end{prop}
\begin{proof}
Recall that $dM_t=\sqrt{\frac{M_t(1-M_t)}{1-t}} \, dB_t$, where the instantaneous quadratic variation $\Sigma_t=\frac{M_t(1-M_t)}{1-t}$ is a martingale over time $[0,1)$. For any $p \in (0,\infty)$, we compute 
\begin{align*}
d  \Sigma_t=& \frac{1-2M_t}{1-t} \, dM_t, \\
d \langle \Sigma \rangle_t =&\frac{(1-2M_t)^2 \Sigma_t}{(1-t)^2} \, dt ,  \\
d(\Sigma_t)^p=& p (\Sigma_t)^{p-1} d \Sigma_t+ \frac{p(p-1)}{2}\frac{\Sigma_t^{p-1}(1-2M_t)^2}{(1-t)^2} \, dt.
\end{align*}
Taking $p=1.5$, with some positive constant $c_{1.5}$ 
$$
\mathbb E[\Sigma_t^{1.5}] \leq  \Sigma_0^{1.5}+c_{1.5}\mathbb E \left[\int_0^t \frac{\Sigma_u^{0.5}}{(1-u)^2}   \, du \right]=\Sigma_0^{1.5}+c_{1.5}  \int_0^t \frac{\mathbb E[\Sigma_u^{0.5}] }{(1-u)^2} \,du, \quad \forall t <1.$$
Therefore, we get that 
\begin{align*}
    \int_0^{1-\epsilon} \mathbb E[\Sigma_t^{1.5}] \, dt & =C + c_{1.5} \int_0^{1-\epsilon} \frac{\mathbb E [\Sigma_t^{0.5}]}{(1-t)^2}(1-t -\epsilon) \, dt  \\
    &= C + c_{1.5} \int_0^{1-\epsilon} \frac{\mathbb E [\Sigma_t^{0.5}] }{1-t} \, dt -\epsilon c_{1.5}  \int_0^{1-\epsilon} \frac{\mathbb E [\Sigma_t^{0.5}] }{(1-t)^2} \, dt,
\end{align*}
where $C$ is a positive constant. Therefore to show \eqref{eq:momentbound}, it is sufficient to prove that 
$$\int_0^1 \frac{\mathbb E[\Sigma_t^{0.5}]}{1-t} \, dt<+\infty. $$

Let us transfer $\int_0^{1} \frac{\mathbb E [\sigma_t] }{1-t} \, dt$ to the density estimate of standard neutral Wright-Fisher diffusion. Take $s: [0,\infty) \to [0,1)$, $t \mapsto 1-e^{-t}$, and define $\tilde M_t:=M_{s(t)}$. Then it can be seen that 
\begin{align*}
    d\tilde M_t= \sqrt{\tilde M_t (1-\tilde M_t)}\, d\tilde B_t,
\end{align*}
for some Brownian motion $\tilde B_t$, and $\tilde M_t$ is a standard neutral Wright-Fisher diffusion. Moreover, the instantaneous quadratic variation process of $\tilde M$  is given by $\tilde \Sigma_t= e^{-t}\Sigma_{s(t)}$, and therefore 
\begin{align}\label{eq:standardWF}
   \mathbb E  \int_0^{1} \frac{\sigma_s}{1-s} \, ds&=\mathbb E \int_0^\infty \frac{ \sigma_{s(t)}}{1-s(t)} e^{-t}\,dt  \\
   &= \mathbb E\int_0^\infty e^{t/2} \tilde \sigma_t \,dt=\mathbb E \int_0^{\infty} e^{t/2} \sqrt{\tilde M_t(1-\tilde M_t) }  \, dt. \notag
\end{align}

Denote by $\rho(t,x,y)$ the density of $\tilde M_t$ given $\tilde M_0=x$. By \cite[Chapter 15, Section 13.F]{KaTa81},
\begin{align*}
    \rho(t,x,y)=\frac{1}{y(1-y)}\sum_{n=1}^{\infty} e^{-n(n+1)t} \varphi_{n-1}(x) \varphi_{n-1}(y) n(n+1)(2n+1),
\end{align*}
with $\varphi_n(x)=x(1-x)P^{(1,1)}_{n-1}(1-2x)$, where $P^{(1,1)}_{n-1}$ are the Jacobi orthogonal polynomials normalized so that $P^{(1,1)}_{n-1}(1)=1$; see \cite[Chapter IV]{Sz39}. Note that $|P^{(1,1)}_{n-1}(-1)|=|P^{(1,1)}_{n-1}(1)| \geq |P^{(1,1)}_{n-1}(x)|$, $\forall x \in [0,1]$, and hence with $g(y):=\sqrt{y(1-y)}$, $y \in [0,1]$, we get 
\begin{align*}
    \mathbb E \left[\sqrt{\tilde M_t(1-\tilde M_t)} \, \Big| \, \tilde M_0=x\right]&=\int g(y) \rho(t;x,y) \, dy \\
    &=\sum_{n=1}^{\infty}e^{-n(n+1) t } n(n+1)(2n+1) \varphi_{n-1}(x) \int_0^1 g(y) P_{n-2}^{(1,1)}(1-2y) \,dy \\
    & \leq \sum_{n=1}^{\infty}e^{-n(n+1) t } n(n+1)(2n+1). 
\end{align*}
It is clear now that the integration in \eqref{eq:standardWF} is finite, concluding the argument. 
\end{proof}

\begin{proof}[Proof of Theorem~\ref{thm}]
    According to Proposition~\ref{prop:WF} and Proposition~\ref{prop:comparison}, $v(t,x)=\bar v(t,x)$ and the scaled Wright-Fisher diffusion \eqref{eq:WFdiffusion} is an optimizer of the control problem \eqref{eq:WFvalue}. For the uniqueness, we can argue as just in \cite{BBexciting}. In the end, Lemma~\ref{lem:swplimit} and Proposition~\ref{prop:moment} conclude \eqref{eq:derivativeswp}. 
\end{proof}

\begin{remark}\label{eq:multidim}
    It is not immediately clear how to identify the multidimensional analogue of the reciprocal specific relative entropy. One natural candidate, in the authors' opinion, arises when we consider the quantum (von Neumann) relative entropy between positive semidefinite matrix-valued measures on $[0,1]$, namely $$S(M||N):= \int_0^1tr(M_t\{\log(M_t)-\log(N_t)\}+N_t-M_t)dt,$$
    where $M_t$ is the density of $M$ at the point $t$, and similarly for $N$. Here $\log$ refers to the matrix logarithm. Specializing to the case when $N_t=tI$, corresponding to the quadratic variation of multidimensional Brownian motion $B$, one can postulate $$
    \mathfrak{h}( \mathbb Q \| \mathbb W):=  \frac{1}{2}\mathbb E_{\mathbb Q}[S( \langle X\rangle\,||\, \langle B\rangle)]=\frac{1}{2}\mathbb E_{\mathbb Q}\left[\int_0^1 [tr(\Sigma_t\log(\Sigma_t))+d-tr(\Sigma_t)] \, dt \right]. $$
Subsequently, one can represent a $d$-dimensional win-martingale as a martingale on the sub-probability simplex that terminates in its vertices. Numerical evidence suggests that the multidimensional neutral Wright-Fisher diffusion is not the optimal solution to the associated win-martingale optimization problem. We thank Yuxing Huang for his help on this matter. It begs to ask whether a different candidate for the multidimensional reciprocal specific relative entropy would have the desirable property of having the   multidimensional neutral Wright-Fisher diffusion as the optimal win-martingale.
\end{remark}

\bibliographystyle{siam}
\bibliography{ref}

\end{document}